%% file: main.tex
\documentclass[a4paper,11pt,leqno]{article}
\usepackage[T1]{fontenc}
\usepackage{lmodern}
\usepackage{amsmath,amsfonts,bm, amsthm}
\usepackage{url}
\usepackage{float}
\usepackage{tikz}
\usetikzlibrary{patterns}
\usepackage{comment}
\newtheorem{conjecture}{Conjecture}
\usepackage[backref]{smacros}

\title{On Sets of Monochromatic Objects in Bicolored Point Sets}

\author{
Sujoy Bhore\thanks{Department of Computer Science, Indian Institute of Technology Bombay, Mumbai, India.\\ Email: \textcolor{gray}{sujoy@cse.iitb.ac.in}}
\quad
Konrad Swanepoel\thanks{Department of Mathematics, London School of Economics, London, UK.\\ Email: \textcolor{gray}{konrad.swanepoel@gmail.com}}
}
\date{}

\begin{document}

\maketitle

\begin{abstract}
Let $P$ be a set of $n$ points in the plane, not all on a line, each colored \emph{red} or \emph{blue}. 
The classical Motzkin--Rabin theorem guarantees the existence of a \emph{monochromatic} line. 
Motivated by the seminal work of Green and Tao (2013) on the Sylvester-Gallai theorem, we investigate the quantitative and structural properties of monochromatic geometric objects, such as lines, circles, and conics.

We first show that if no line contains more than three points, then for all sufficiently large $n$ there are at least $n^{2}/24 - O(1)$ monochromatic lines.  
We then show a converse of a theorem of Jamison (1986): Given $n\ge 6$ blue points and $n$ red points, if the blue points lie on a conic and every line through two blue points contains a red point, then all red points are collinear. 
We also settle the smallest nontrivial case of a conjecture of Mili\'cevi\'c (2018) by showing that if we have $5$ blue points with no three collinear and $5$ red points, if the blue points lie on a conic and every line through two blue points contains a red point, then all $10$ points lie on a cubic curve.
Further, we analyze the random setting and show that, for any non-collinear set of $n\ge 10$ points independently colored red or blue, the expected number of monochromatic lines is minimized by the \emph{near-pencil} configuration. Finally, we examine monochromatic circles and conics, and exhibit several natural families in which no such monochromatic objects exist.
\end{abstract}

\section{Introduction}
Incidence problems lie at the heart of combinatorial geometry, which fundamentally studies how points and lines in the plane relate to one another.
Among its earliest and most influential results is the \emph{Sylvester–Gallai theorem}, proved by Gallai~\cite{gallai1944solution}, by resolving a question posed nearly forty years earlier by Sylvester~\cite{sylvester1893mathematical} in 1893.

\begin{theorem}[Sylvester--Gallai theorem]\label{thm:SG}
Suppose that $P$ is a finite set of points in the plane, not all on one line.  
Then there exists an \textbf{ordinary line} spanned by $P$, i.e., a line containing exactly two points of $P$.
\end{theorem}

Several elegant proofs of the Sylvester–Gallai theorem are known.
For instance, Melchior’s classical proof~\cite{melchior1941vielseite}, based on projective duality and Euler’s formula, shows that every non-collinear set of $n$ points spans at least three ordinary lines.
Subsequent work focused on obtaining stronger quantitative bounds: Motzkin~\cite{motzkin1951lines} established a lower bound of order $\sqrt{n}$; Kelly and Moser~\cite{kelly1958number} improved this to $3n/7$; and Csima and Sawyer~\cite{csimasawyer1993} further strengthened it to $6n/13$ for $n>7$, building on ideas of Hansen~\cite{hansen1981contributions}.
Comprehensive discussions of these results appear in the surveys of Borwein and Moser~\cite{borwein1990survey}, and Pach and Sharir~\cite{pach2009combinatorial}.
A major breakthrough came with the work of Green and Tao~\cite{greentao2013}, who proved that any sufficiently large non-collinear set of $n$ points in the plane determines at least $n/2$ ordinary lines, which confirmed the long-standing \emph{Dirac–Motzkin conjecture} for all large $n$, and characterized all extremal configurations. 

\paragraph{Colorful Sylvester--Gallai (Motzkin-Rabin Theorem).}
Given a finite configuration of points in the plane, it is natural to ask whether certain geometric or combinatorial properties are preserved under colorings of the points.
In particular, one may consider a two-coloring of a finite, non-collinear set of points, say, coloring each point either \emph{red} or \emph{blue}, and ask whether the existence of an ordinary line, as guaranteed by the \emph{Sylvester–Gallai theorem}, can still be ensured in a colorful instance of the point set. That is, can we always find an ordinary line determined by two points of one color? 
The answer to that question is negative, as can be seen from a simple example. Consider two intersecting lines $\ell_1$ and $\ell_2$: place all red points on $\ell_1$, all blue points on $\ell_2$, and say the color of the intersection point of $\ell_1$ and $\ell_2$ is red. In this configuration, every line containing at least two points of one color also contains a further point; hence, no \emph{ordinary monochromatic line} exists.
This observation naturally leads to the following question: even if no ordinary monochromatic line exists, must there always exist at least one line containing points of only one color?
The following classical result, known as the \emph{Motzkin–Rabin theorem} (see~\cite{borwein1990survey, erdHos1996extremal}), provides an affirmative answer to this question and may be viewed as a natural colored analogue of the Sylvester–Gallai theorem.

\begin{theorem}[Motzkin--Rabin theorem]\label{thm:MR}
Let $S$ be a finite, non-collinear set of points in the plane, each colored red or blue.  
Then there exists a line $\ell$ passing through at least two points of $S$, all of which have the same color.
\end{theorem}

There are two essentially different proofs of the Motzkin--Rabin theorem known in the literature, both established via the projective dual form of the statement.  
In the projective plane, the dual version asserts that for any finite nonconcurrent set of lines, each colored red or blue, there exists a point where all intersecting lines share the same color.  
The first proof, due to Motzkin (see~\cite{grunbaum1999monochromatic}) and published in~\cite{borwein1990survey, erdHos1996extremal}, is based on a minimal-area argument.  
The second proof, due to Chakerian~\cite{chakerian1970sylvester} (see also~\cite{edmonds1980solution}), derives the result from a version of Cauchy’s lemma on planar graphs with colored edges, which itself follows essentially from Euler’s formula. Later, Pretorius and Swanepoel~\cite{pretorius2004algorithmic} gave an algorithmic proof for Motzkin--Rabin theorem (see also~\cite{coxeter1948problem}). 

Motivated by the seminal work of Green and Tao~\cite{greentao2013} and several other foundational studies~\cite{burr1974orchard, melchior1941vielseite, kelly1958number, van2007enumerating, elekes2000convexity, csimasawyer1993, crowe1968sylvester} (see also Section~\ref{related-work}), we turn to structural questions related to the Motzkin–Rabin theorem.
Although the colorful version of the Sylvester–Gallai problem may appear quite different, especially in view of the discussion above on ordinary lines, it is natural to ask whether similar combinatorial phenomena still occur.
In particular, we are interested in understanding to what extend the geometric and algebraic structure that appears in the uncolored case continues to manifest in its colored counterpart.

\begin{quote}
\begin{center}
\emph{Given a finite set of red and blue points in the plane, does there necessarily exist a large number of monochromatic lines? More generally, what happens if we replace lines with other geometric objects, such as circles or conics, and do similar phenomena persist?}
\end{center}
\end{quote}

\subsection{Our Contributions}
In this work, we provide a structural and quantitative understanding of the above question. As discussed earlier, without imposing any restriction on collinearities, one cannot hope to obtain non-trivial lower bounds on the number of monochromatic lines. We show that under the assumption that no line contains more than three points, every two-colored set of $n$ points in the plane (not all on a line) determines at least $n^{2}/24 - O(1)$ monochromatic lines for all sufficiently large $n$. Moreover, we establish a matching structural characterization: if the number of monochromatic lines is at most $n^{2}/24 + K n$, then, after adding or deleting only $O(K)$ points, the set is contained in a coset of a finite subgroup of a cubic curve that is either smooth or acnodal. This yields a near-complete description of all near-extremal configurations 
(Theorem~\ref{thm:1}).

We next establish a converse to an earlier theorem of Jamison~\cite{jamison1986few}, by showing that if $n\ge 6$ blue points lie on a conic and every line determined by two blue points contains a red point, then all red points must be collinear (Theorem~\ref{thm:conic-line}). For the extremal case $n=5$, we further prove that the same condition forces all $2n$ points to lie on a cubic curve (Theorem~\ref{thm:Milicevic}), thereby resolving the smallest nontrivial case of a conjecture of Mili\'cevi\'c~\cite{milicevic2018}. To the best of our knowledge, this converse direction has not been previously examined.

Next, in Theorem~\ref{thm:random}, we show that if each point of a non-collinear set of $n \ge 10$ points is independently colored red or blue uniformly at random, then the configuration that minimizes the expected number of monochromatic lines is the \emph{near-pencil} (that is, $n-1$ points on a line and one point off the line; see Figure~\ref{fig:near-pencil}).

Already, in 1951, Motzkin \cite{motzkin1951lines} proved that for any finite set of points in the plane, not all lying on a line or a circle, determines an \emph{ordinary circle}, namely, a circle incident to exactly three points of the set.
Later, Elliott~\cite{elliott1967} proved in 1967 
quadratic lower bound on the number of ordinary circles.
One might expect an analogous phenomenon to hold in the monochromatic setting. Surprisingly, we show that this intuition fails: there exists an infinite family of configurations with neither color class contained in a circle or a line and with no monochromatic circle. 
We further demonstrate that a similar failure of ordinary monochromatic objects persists for conics as well. Specifically, we show that there exists an infinite family of configurations with neither color class contained in a conic and with no monochromatic conic.

\section{Monochromatic Lines}
We first investigate how the assumption that at most $k$ points lie on any line affects the number of monochromatic lines in a set of $n$ points, each colored red or blue.
If $k=2$, then the number of monochromatic lines equals $\binom{b}{2}+\binom{r}{2}$, where $b$ is the number of blue points and $r$ the number of red points, and this is minimized when $b$ and $r$ are equal.
Thus, the first non-trivial case is $k=3$, for which we first describe an example.
\begin{example}\label{ex:cubic}
Consider any smooth irreducible cubic curve $\gamma$ in the projective plane.
Note that no $4$ points on $\gamma$ are collinear.
It is well known that there is a group $(\gamma,\oplus)$ such that three points of $\gamma$ are collinear if and only if they sum to the identity in the group.
This group contains finite subgroups isomorphic to $\mathbb{Z}_n$ for all $n$.
If $n$ is even, let the blue points be the points corresponding to $\{0,2,4,\dots,n-2\}$, and the red points $\{1,3,5,\dots,n-1\}$.
Then the monochromatic lines in this set of $n$ points will always be blue, and will correspond to the total number of lines spanned by the $n/2$ blue points, which can be calculated as $n^2/24 + n/4 - O(1)$.
If we remove or add $O(K)$ points to this example, then the number of monochromatic lines will be $n^2/24 +O(Kn)$.
\end{example}

In the following theorem, we obtain a lower bound that is tight up to an additive $n/4$ term compared to the example above.

\begin{theorem}\label{thm:1}
In a set of $n$ points in the plane with at most $3$ points on a line, and with each point colored red or blue, there are at least $n^2/24-O(1)$ monochromatic lines if $n$ is sufficiently large.
Furthermore, for each $K\geq 1$ there exists $n_0$ such that for all $n\geq n_0$, if there are at most $n^2/24 +Kn$ monochromatic lines then, up to adding or removing $O(K)$ points, 
the set forms a coset of a finite subgroup of a cubic curve that is either smooth or acnodal.
\end{theorem}

To prove this theorem, first we show the following general result on the number of monochromatic lines in abstract geometries with at most three points on a line.
Then, we combine it with Green and Tao's structure theorem for sets with few ordinary lines \cite{greentao2013} (Lemma~\ref{lemma:GT} below), to obtain the result. 

An \emph{abstract geometry} consists of a set of points together with a collection of subsets of the points, called lines, with the property that each line contains at least two points, and for any two points there is a unique line containing them.

\begin{lemma}\label{lemma:1}
In an abstract geometry with $n$ points and at most $3$ points on each line, if the points are each colored red or blue, then the number of monochromatic lines is at least $\frac{n^2}{24}-\frac{n}{6} + \frac{t_2}{6}$, where $t_2$ is the number of lines containing exactly $2$ points.   
\end{lemma}
\begin{proof}
Let $b$ denote the number of blue points, $r$ the number of red points, and $t_{i,j}$ the number of lines containing $i$ blue points and $j$ red points, for each $i,j\geq 0$.
Then $b+r=n$, and by counting the number of pairs of blue points, the number of pairs of red points, and the number of blue-red pairs, we obtain
\begin{align}
    \binom{b}{2} &= 3t_{3,0} + t_{2,1} + t_{2,0}, \label{count1} \\
    \binom{r}{2} &= 3t_{0,3} + t_{1,2} + t_{0,2}, \label{count2} \\
    br &= 2t_{2,1} + 2t_{1,2} + t_{1,1}. \label{count3}
\end{align}
By eliminating $t_{2,1}$ and $t_{1,2}$ from \eqref{count1}, \eqref{count2} and \eqref{count3}, we obtain that the number of monochromatic lines is
\begin{align*}
    & \mathrel{\phantom{=}} t_{3,0} + t_{0,3} + t_{2,0} + t_{0,2}\\
    &= \frac{1}{6}\left(b^2-br+r^2-b-r+4t_{2,0}+4t_{0,2}+t_{1,1}\right)\\
    &= \frac{1}{6}\left(\frac{1}{4}(b+r)^2 - (b+r) +\frac{3}{4}(b-r)^2+4t_{2,0}+4t_{0,2}+t_{1,1}\right)\\
    &\geq \frac{1}{6}\left(\frac{1}{4}n^2 - n + t_2\right).\qedhere
\end{align*} 
\end{proof}

\begin{lemma}[Green--Tao {\cite[Theorem~1.5]{greentao2013}}]\label{lemma:GT}
For any $K \geq 1$ there exists $n_0$ such that for any set of $n\geq n_0$ points in the plane with at most $Kn$ ordinary lines, the set differs by at most $O(K)$ points from one of the following:
\begin{enumerate}
    \item $n+O(K)$ points on a line,
    \item $n/2+O(K)$ points on a conic and $n/2+O(K)$ points on a straight line not intersecting the conic,
    \item A coset of a subgroup of cardinality $n+O(K)$ of the group on the non-singular points of an irreducible cubic curve that is smooth or acnodal.
\end{enumerate} 
\end{lemma}

\begin{lemma}\label{lemma:2}
For any $K \geq 1$ there exists $n_0$ such that any set of $n\geq n_0$ points that differs by at most $O(K)$ points from a coset of a finite subgroup on an irreducible cubic curve determines at least $n-O(K)$ ordinary lines.
\end{lemma}
\begin{proof}
It is well known (and explained in \cite{greentao2013}) that a coset $H\oplus x$ of a finite subgroup of order $m$ of the group on the cubic curve $\gamma$ is a set of $m$ points with at least $m-O(1)$ ordinary lines, and they are all tangent to $\gamma$.
It is also a well-known classical fact (the Pl\"ucker formulas) that through any fixed point $p$ in the plane, there are at most $6$ lines through $p$ tangent to $\gamma$ if $p\notin\gamma$, and at most $4$ tangent lines through $p$ if $p\in\gamma$.
Thus, if we remove a point in $H\oplus x$, we destroy at most $4$ of the $m-O(1)$ many tangent lines, and for each point that we add, we destroy at most $6$ of the original ordinary lines.
Thus, after removing and adding $O(K)$ points, at least $n-O(K)$ of the original ordinary lines of the coset are left untouched.
\end{proof}

We now proceed to the proof of the theorem.

\begin{proof}[Proof of Theorem~\ref{thm:1}]
We first show the second part of the theorem.
Let $S$ be a set of $n$ points in the plane, each colored red or blue, with at most $3$ points on a line and with at most $n^2/24+Kn$ monochromatic lines.
By Lemma~\ref{lemma:1}, the number of monochromatic lines is at least $n^2/24 -n/6 + t_2/6$, so \[n^2/24 + Kn \geq n^2/24 -n/6 + t_2/6,\] 
and $t_2 \leq (6K+1)n$.
By Lemma~\ref{lemma:GT}, $S$ differs by $O(K)$ points from a coset of a finite subgroup of a cubic curve that is smooth or acnodal.

It remains to show the first part of the theorem.
Suppose that $S$ has at most $n^2/24+n$ monochromatic lines.
If we let $K=1$ in the above, then we deduce that $S$ differs by $O(1)$ points from a coset, and by Lemma~\ref{lemma:2}, $t_2\geq n - O(1)$,
from which it follows (again using Lemma~\ref{lemma:1}) that the number of monochromatic lines is at least $n^2/24-O(1)$.
\end{proof}

For any $2$-coloring of the construction on an irreducible cubic curve in the third case of Lemma~\ref{lemma:GT}, it can be checked that there are at least $n^2/24$ monochromatic lines, with equality when the number of red and blue points is  exactly the same.
Thus, in this case, there is a subquadratic randomised algorithm for finding such a \emph{monochromatic} line, merely by guessing and checking.
This holds for any configuration with superlinearly many monochromatic lines.

Instead of assuming that there are at most $3$ points on a line, we can assume there are at most $k$ points on a line for some small value of $k$, and try to do the same counting as above.
We again obtain a quadratic number of monochromatic lines, as long as the number of blue and red points are sufficiently unbalanced.
The exact statement is in Corollary~\ref{cor:prep-k-line}, and needs the following technical result.

\begin{proposition}\label{prep-k-line}
In a set of $n$ points in the plane with at most $k$ points on a line, the number of monochromatic lines is at least \[\frac{b^2-(k-2)br+r^2-n}{k(k-1)},\] where $b$ denotes the number of blue points and $r$ the number of red points.
\end{proposition}

\begin{proof}
Similar to \eqref{count1}, \eqref{count2}, \eqref{count3}, we have
\begin{align}
    \binom{b}{2} &= \sum_{2\leq i+j\leq k}\binom{i}{2}t_{i,j}\\
    \binom{r}{2} &= \sum_{2\leq i+j\leq k}\binom{j}{2}t_{i,j}\\
    br &= \sum_{2\leq i+j\leq k}ijt_{i,j}
\end{align}
Then (with $\lambda > 0$ to be fixed later)
\begin{align*}
&\binom{b}{2}+\binom{r}{2}-\lambda br\\
&= \sum_{2\leq i+j\leq k}\left(\binom{i}{2}+\binom{j}{2}-\lambda ij\right)t_{i,j}\\
&= \sum_{2\leq i\leq k}\binom{i}{2}t_{i,0} + \sum_{2\leq j\leq k}\binom{j}{2}t_{0,j} + \sum_{i+j\leq k; i,j\geq 1} \frac12\left(i^2-i+j^2-j-2\lambda i j\right)t_{i,j}.
\end{align*}
Note that $i^2-i+j^2-j-2\lambda ij$ is maximized under the constraints $i,j\geq 1$ and $i+j\leq k$ when $\{i,j\}=\{k-1,1\}$.
Thus
\begin{align*}
    & \binom{b}{2}+\binom{r}{2}-\lambda br\\
    & \leq \binom{k}{2}\left(\sum_{2\leq i\leq k}t_{i,0} + \sum_{2\leq j\leq k}t_{0,j}\right) + \frac12(k-1)(k-2-2\lambda)\sum_{i+j\leq k; i,j\geq 1}t_{i,j}\\
    &\leq \binom{k}{2}\left(\sum_{2\leq i\leq k}t_{i,0} + \sum_{2\leq j\leq k}t_{0,j}\right)
\end{align*}
if $k-2-2\lambda\leq 0$.
Thus, we choose $\lambda = (k-2)/2$, and we obtain that the number of monochromatic lines is 
\[\sum_{2\leq i\leq k}t_{i,0} + \sum_{2\leq j\leq k}t_{0,j} \geq \binom{k}{2}^{-1}\left(\binom{b}{2}+\binom{r}{2}-\frac{k-2}{2}br\right) = \frac{b^2-(k-2)br+r^2-n}{k(k-1)}.\]
\end{proof}

\begin{corollary}\label{cor:prep-k-line}
In a set of $n$ points in the plane with at most $k$ points on a line, if one of the two color classes have at most $cn$ points, where $0 < c < 2/(k+\sqrt{k^2-4k})$, then there are $\Omega_k(n^2)$ monochromatic lines.
\end{corollary}
\begin{proof}
Denote the number of blue points by $b$, and the number of red points by $r$, and assume without loss of generality that $b\leq cn$.
Using $b+r=n$, we can write $b^2-(k-2)br+r^2 = (kc^2-kc+1)n^2$.
The quadratic polynomial $kc^2-kc+1$ is decreasing on the interval $0 < c < 2/(k+\sqrt{k^2-4k})$ with a root at $2/(k+\sqrt{k^2-4k})$.
Therefore, $C:=kc^2-kc+1 > 0$
By Proposition~\ref{prep-k-line}, the number of monochromatic lines is at least $\frac{Cn^2-n}{k(k-1)} = \Omega(n^2)$.
\end{proof}

Proposition~\ref{prep-k-line} established a quantitative lower bound and naturally leads the following structural question:

\begin{question}
What is the structure of a $2$-colored point set in which there are very few monochromatic lines?
\end{question}

It is easy to construct an example of a set with only one monochromatic line.
\begin{example}\label{ex:one-monochromatic-line}
Fix a line $\ell$ and choose any number $b$ of blue points not on $\ell$.
For each line through two blue points, color its intersection point with $\ell$ red.
Add an arbitrary number of red points on $\ell$.
Then there is no blue monochromatic line, and $\ell$ is a monochromatic line containing all the red points.
\end{example}
If we choose the blue points in the above example generically, we end up with a minimum of $\binom{b}{2}$ red points, and then the colors are very unbalanced.
At the other extreme, it is possible in the above example to choose the blue points such that we end up with the same number of red points as blue points.
For later reference, we describe this special case next.
\begin{example}[B\"or\"oczky \cite{crowe1968sylvester}]\label{ex:boroczky}
Consider the vertices of a regular $k$-gon on a circle.
The lines through pairs of vertices (edges and diagonals of the $k$-gon), as well as the lines tangent to the circle at a vertex, fall in $k$ parallel classes, corresponding to $k$ points at infinity.
This example has $n=2k$ points and there are exactly $k=n/2$ ordinary lines, namely the tangent lines at the $n$ vertices.
If we color the $k$ points on the circle blue and the $k$ points on the line red, then we obtain the same number of blue and red points, and the line at infinity is the only monochromatic line.
See the case $k=5$ in Figure~\ref{fig:boroczky}, which has been drawn in perspective such that the line at infinity is visible.
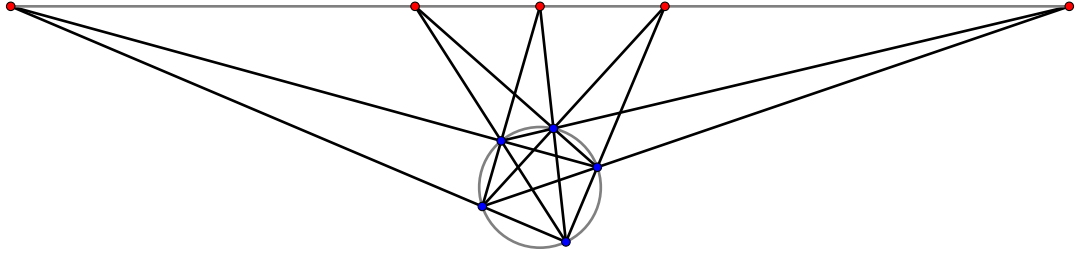
\begin{figure}
\centering
\begin{tikzpicture}[line cap=round,line join=round,scale=0.8]
\draw [color=gray,line width=1pt] (0,0) circle (1cm);
\draw [line width=1pt] (0.9428090415820634,0.3333333333333334)-- (-2.054972593492032,3);
\draw [line width=1pt] (0.42657694360473053,-0.9044512762912365)-- (2.05497259349203,3);
\draw [line width=1pt] (0.42657694360473053,-0.9044512762912365)-- (-8.705003597931286,3);
\draw [line width=1pt] (0.42657694360473053,-0.9044512762912365)-- (-2.054972593492032,3);
\draw [line width=1pt] (0.42657694360473053,-0.9044512762912365)-- (0,3);
\draw [line width=1pt] (0.9428090415820634,0.3333333333333334)-- (-8.705003597931286,3);
\draw [line width=1pt] (-0.9486077528733273,-0.31645431137624985)-- (8.705003597931283,3);
\draw [line width=1pt] (-0.9486077528733273,-0.31645431137624985)-- (0,3);
\draw [line width=1pt] (-0.9486077528733273,-0.31645431137624985)-- (2.05497259349203,3);
\draw [line width=1pt] (-0.637787785600776,0.7702121399578551)-- (8.705003597931283,3);
\draw [color=gray,line width=1pt] (-8.705003597931286,3)-- (8.705003597931283,3);
\draw [fill=red] (-8.705003597931286,3) circle (2pt);
\draw [fill=red] (0,3) circle (2pt);
\draw [fill=red] (8.705003597931283,3) circle (2pt);
\draw [fill=red] (-2.054972593492032,3) circle (2pt);
\draw [fill=red] (2.05497259349203,3) circle (2pt);
\draw [fill=blue] (0.22121476756733635,0.9752251158630654) circle (2pt);
\draw [fill=blue] (0.9428090415820634,0.3333333333333334) circle (2pt);
\draw [fill=blue] (0.42657694360473053,-0.9044512762912365) circle (2pt);
\draw [fill=blue] (-0.9486077528733273,-0.31645431137624985) circle (2pt);
\draw [fill=blue] (-0.637787785600776,0.7702121399578551) circle (2pt);
\end{tikzpicture}
\caption{B\"or\"oczky example: Five blue points on a circle and five red points on a line. There is just one monochromatic line.}\label{fig:boroczky}
\end{figure}
\end{example}

We conjecture that the B\"or\"oczky example is extreme, in the sense that if we assume that fewer than half of the points are on a line, then there will be many monochromatic lines.
\begin{conjecture}
For any $\delta > 0$ there exists $\epsilon > 0$ such that in a non-collinear set of $n$ two-colored points, if the number of points on a line is at most $(1/2-\delta)n$, then the number of monochromatic lines is at least $\epsilon n^2$.  
\end{conjecture}

Our next result shows that Example~\ref{ex:one-monochromatic-line} describes all possible two-colored sets with only one monochromatic line.
It implies a slight strengthening of the Motzkin–Rabin theorem, namely that if both color classes are non-collinear, then more than one monochromatic line must exist.

\begin{theorem}\label{thm:2.11}
Let $S$ be a finite set of points, each colored blue or red.
Suppose that there is just one monochromatic line $\ell$.
Then the two color classes are $S\cap\ell$ and $S\setminus\ell$.
\end{theorem}
\begin{proof}
We prove the projective dual statement, namely that in a finite set of lines in the plane, each colored blue or red, if there is just one monochromatic intersection point, of blue lines, say, then all blue lines intersect in that point.

Let $P$ be the blue intersection point, and suppose that not all the blue lines pass through $P$.
After a projective transformation, $P$ is on the line at infinity and the line at infinity is not one of the given lines.
Then all lines through $P$ are parallel, and again without loss of generality, we can assume that all blue lines through $P$ are vertical lines.
Then our assumption becomes that not all blue lines are vertical.
Let $b_1$ be a non-vertical blue line (Figure~\ref{fig:MR}).
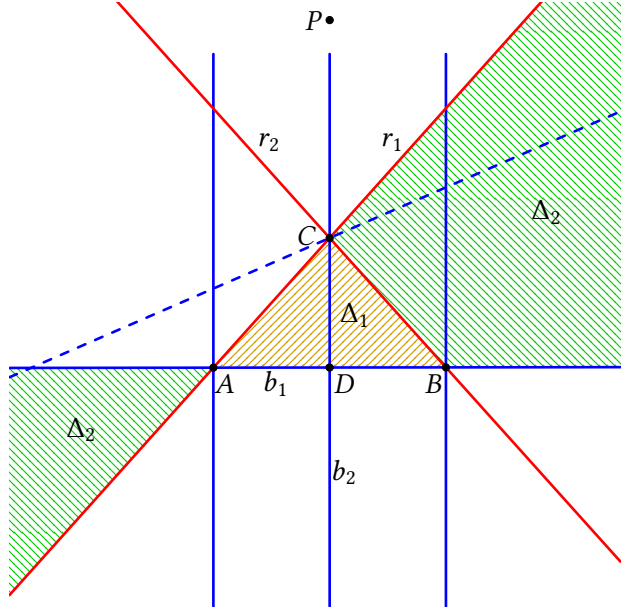
\begin{figure}
\centering
\definecolor{qqzzqq}{rgb}{0,0.8,0}
\definecolor{zzwwqq}{rgb}{0.8,0.6,0}
\begin{tikzpicture}[line cap=round,line join=round,scale=0.9]
\clip(-3.7506204491979935,-2.638050954481335) rectangle (5.181992622791327,6.258213742850853);
\fill[line width=0pt,color=zzwwqq,fill=zzwwqq,pattern=north east lines,pattern color=zzwwqq] (0.93,2.79) -- (-0.77,0.8793002422290823) -- (2.63,0.8836003812859525) -- cycle;
\fill[line width=0pt,color=qqzzqq,fill=qqzzqq,pattern=north west lines,pattern color=qqzzqq] (5.18130671011677,7.568218059547052) -- (0.93,2.79) -- (2.63,0.8836003812859525) -- (7.69,0.89) -- cycle;
\fill[line width=0pt,color=qqzzqq,fill=qqzzqq,pattern=north west lines,pattern color=qqzzqq] (-6.620234094422936,0.8719011774875347) -- (-6.223076923076924,-5.249636673082247) -- (-0.77,0.8793002422290823) -- cycle;
\draw [line width=1pt,color=blue] (-0.77,-2.638050954481335) -- (-0.77,5.5);
\draw [line width=1pt,color=blue] (0.93,-2.638050954481335) -- (0.93,5.5);
\draw [line width=1pt,color=blue] (2.63,-2.638050954481335) -- (2.63,5.5);
\draw [line width=1pt,color=blue,domain=-3.7506204491979935:5.181992622791327] plot(\x,{(--9.794111830156845--0.014071834503441738*\x)/11.126207008413832});
\draw [line width=1pt,color=red,domain=-3.7506204491979935:5.181992622791327] plot(\x,{(--2.9660492252730464--1.9106997577709177*\x)/1.7});
\draw [line width=1pt,color=red,domain=-3.7506204491979935:5.181992622791327] plot(\x,{(-6.515951645404065--1.9063996187140475*\x)/-1.7});
\draw [line width=1pt,dash pattern=on 3pt off 4pt,color=blue,domain=-3.7506204491979935:5.181992622791327] plot(\x,{(--10.401630747386394--1.9140718345034418*\x)/4.3662070084138325});
\draw[color=black] (0.15,0.65) node {$b_1$};
\draw[color=black] (1.15,-0.65) node {$b_2$};
\draw [fill=black] (-0.77,0.8793002422290823) circle (1.5pt);
\draw[color=black] (-0.6,0.65) node {$A$};
\draw [fill=black] (0.93,2.79) circle (1.5pt);
\draw[color=black] (0.6,2.83) node {$C$};
\draw[color=black] (1.85,4.2) node {$r_1$};
\draw [fill=black] (2.63,0.8836003812859525) circle (1.5pt);
\draw[color=black] (2.45,0.65) node {$B$};
\draw[color=black] (0.05,4.2) node {$r_2$};
\draw[color=black] (1.3,1.6556008018919623) node {$\Delta_1$};
\draw[color=black] (4.1,3.173145443827075) node {$\Delta_2$};
\draw[color=black] (-2.7,-0.0341789227355431) node {$\Delta_2$};
\draw [fill=black] (0.93,0.8814503117575174) circle (1.5pt);
\draw[color=black] (1.15,0.65) node {$D$};
\draw [fill=black] (0.93,6) circle (1.5pt);
\draw [color=black] (0.7,6) node {$P$};
\end{tikzpicture}
\caption{Illustration of Theorem~\ref{thm:2.11}}\label{fig:MR}
\end{figure}
Then $b_1$ intersects the vertical blue lines (of which there are at least two) in points that are not at infinity.
Through each of these intersection points there has to pass a red line.
Let $r_1$ and $r_2$ be two such red lines, intersecting $b_1$ in $A$ and $B$, respectively.
They are not vertical and not parallel to $b_1$, so they, together with $b_1$ form four triangles in the projective plane (each with vertices $A,B,C$), with only one of the triangles containing $P$.
(In Figure~\ref{fig:MR}, two of the triangles not containing $P$, $\Delta_1$ and $\Delta_2$, are indicated.)
There has to be a blue line $b_2$ through the intersection $C$ of $r_1$ and $r_2$, and $b_2$ will cut at least one of the triangles bounded by $r_1,r_2,b_1$ that does not contain $P$ into two smaller triangles.
In Figure~\ref{fig:MR}, $b_2$ cuts triangle $\Delta_1$ into triangles $ADC$ and $BDC$.
(The dashed blue line in the figure shows the case where $b_2$ cuts $\Delta_2$.)
Then there has to be a red line $r_3$ that passes through the intersection point $D$ of $b_1$ and $b_2$ that will cut either triangle $ADC$ or $BDC$.
Since this process can be repeated indefinitely, we obtain a contradiction.

Thus, our original assumption that not all the blue lines pass through $P$ must be false.
\end{proof}

The B\"or\"oczky example (Example~\ref{ex:boroczky}) lies on the union of a conic and a line, which is a (reducible) cubic curve.
Since there are no blue monochromatic lines, the line through any two blue points passes through a red point.
Example~\ref{ex:cubic} gives another such an example, with the points all lying on an irreducible cubic curve.
Mili\'cevi\'c \cite{milicevic2018} conjectured that all such constructions must lie on a cubic curve.
\begin{conjecture}[Mili\'cevi\'c \cite{milicevic2018}]
Let $n$ blue points, no three on a line, and $n$ red points, disjoint from the blue points, be given.
If the line through any two blue points contains a red point, then all $2n$ points lie on a cubic curve.
\end{conjecture}

We show that this conjecture holds for $n=5$.

\begin{theorem}\label{thm:Milicevic}
Given a set of $5$ blue points, no three on a line, and a set of $5$ red points, disjoint from the blue points, such that the line through any two blue points contains a red point, then it follows that all $10$ points lie on a cubic curve.
\end{theorem}

\begin{proof}
Denote the set of blue points by $B$ and the set of red points by $R$.
Exactly as in the proof of Lemma~\ref{lemma:3}, if we define a bipartite graph $G$ with parts $B$ and $R$ by connecting $b\in B$ and $r\in R$ if the line through $b$ and $r$ does not contain any other blue points, then $G$ is a perfect matching.
Thus, for each red point $r$ there are exactly three lines through $r$ and a blue point, with two of them containing $2$ blue points, and one containing only one blue point, which is the neighbour of $r$ in $G$.
Thus, each red point $r$ determines a matching $M_r$ of two edges on the set of blue points.
Since there is a red point on the line through any two blue points, it follows that none of the matchings $M_r$ have an edge in common.
In addition, no two matchings have the same set of $4$ vertices as endpoints.
Then it can be checked that there is just one possibility for the matchings, namely we can label the blue points as $b_i$ and the red points as $r_i$, $i=1,\dots,5$, such that $M_{r_i}=\{b_i b_{i+1}, b_{i+2}b_{i+4}\}$, where indices are modulo $5$.
Next, let $\gamma$ be a cubic that passes through the $9$ points $b_1,\dots,b_5,r_1,\dots,r_4$.
We will show that $r_5\in\gamma$.

First, suppose that $\gamma$ is irreducible.
We will use the group $(\gamma^*,\oplus)$ on the set of its smooth points.
Since each of the points $b_1,\dots,b_5,r_1,\dots,r_4\in \gamma$ lies on a line through two others in this set, none of these points can be the singular point (if any) of $\gamma$.
From the collinearities, we obtain the following relations in the group:
$b_2\oplus b_5\oplus r_1=o=b_3\oplus b_4\oplus r_1$, hence
$b_2\oplus b_5=b_3\oplus b_4$, and similarly,
$b_3\oplus b_1=b_4\oplus b_5$,
$b_4\oplus b_2=b_5\oplus b_1$,
$b_5\oplus b_3=b_1\oplus b_2$.
Let $r_5'=\ominus b_1\ominus b_4$ and
$r_5''=\ominus b_2\ominus b_3$.
Thus, $r_5'$ is the third point of intersection of the line $b_1b_4$ with $\gamma$, and
$r_5''$ is the third point of intersection of the line $b_2b_3$ with $\gamma$.
From the above, we have $b_1\ominus b_2 = b_4\ominus b_5 = b_2\ominus b_3 = b_5\ominus b_1 = b_3\ominus b_4$, hence $r_5'=r_5''$.
Therefore, the line $b_2b_3$ and $b_1b_4$ intersect in a point on $\gamma$, hence $r_5\in\gamma$.

Next, consider the case where $\gamma$ is the union of a conic $C$ and a line $\ell$.
Suppose that there is a blue point on $\ell$.
Since no line contains a blue point and more than one red point, there is at most one red point on $\ell$, hence at least $3$ red points on $C$.
There are also at most $2$ blue points on $\ell$, hence at least $3$ blue points on $C$.
The three lines through these blue points have to contain distinct red points, but they all have to lie on $\ell$, a contradiction.

Thus, all $5$ blue points are on $C$.
It follows that no red point is on $C$, hence $r_1,\dots,r_4\in\ell$.
We now use the natural group associated to this cubic.
Then there are bijections $f: C\cap\gamma^*\to G$ and $g:\ell\cap\gamma^*\to G$ from the conic and the line to some abelian group $(G,\oplus)$ such that non-singular points $a,b\in C$ and $c\in\ell$ are collinear iff $f(a)\oplus f(b)\oplus g(c) = 0$.
(We have $G$ is isomorphic to the circle group if $\ell$ is disjoint from $C$, isomorphic to $(\mathbb{R},+)$ is $\ell$ is tangent to $C$, and isomorphic to $(\mathbb{R}^*,\cdot)$ if $\ell$ intersects $C$ in two points.)
As before, none of the $9$ points $b_1,\dots,b_5,r_1,\dots,r_4$ are singular, and we obtain $f(b_2)\oplus f(b_5)=f(b_3)\oplus f(b_4)$, 
$f(b_3)\oplus f(b_1)=f(b_4)\oplus f(b_5)$,
$f(b_4)\oplus f(b_2)=f(b_5)\oplus f(b_1)$,
$f(b_5)\oplus f(b_3)=f(b_1)\oplus f(b_2)$,
and as before, we can then show that $r_5=g^{-1}(\ominus f(b_1)\ominus f(b_4)) = g^{-1}(\ominus f(b_2)\ominus f(b_3))$.

Finally, consider the case where $\gamma$ is a union of at most three lines.
Since there are at most two blue points on a line, we must have that $\gamma$ is a union of three distinct lines $\ell_1,\ell_2,\ell_3$.
Then two of the lines, say $\ell_1$ and $\ell_2$, each contains at two blue points, and $\ell_3$ contains one blue point.
Both $\ell_1$ and $\ell_2$ contain at most one red point, so $\ell_3$ has to contain at least two red points, a contradiction.
%\qed
\end{proof}

Years before Mili\'cevi\'c formulated his conjecture, Jamison \cite{jamison1986few} already proved that if the red points lie on a line, then the blue points have to lie on a conic, which affirms a special case of the conjecture.
To the best of our knowledge, the converse has not been considered before.
This is our next theorem.

\begin{theorem}\label{thm:conic-line}
Let $n\geq 6$ blue points on a conic, and $n$ red points, disjoint from the blue points, be given.
If the line through any two blue points contains a red point, then the red points lie on a line.
\end{theorem}
This theorem cannot be extended to the case $n=5$: 
For a counterexample, see Figure~\ref{fig:10points}.
This example is related to the B\"or\"oczky example of $10$ points shown in Figure~\ref{fig:boroczky}.
It consists of the vertices $b_1,b_2,b_3,b_4,b_5$ of a regular pentagon inscribed in a circle, together with the point $r_1$ where the diagonals $b_2b_4$ and $b_3b_5$ intersect, the point $r_3$ where the diagonals $b_2b_5$ and $b_1b_4$ intersect, the point $r_4$ where the edges $b_1b_5$ and $b_2b_3$ intersect, the point $r_5$ where the edges $b_1b_2$ and $b_3b_4$ intersect, and the point $r_2$ at infinity where the parallel lines $b_4b_5$, $r_1r_3$, $b_1b_3$, $r_4r_5$ intersect.
The depiction in Figure~\ref{fig:10points} is drawn in perspective such that $r_2$ does not lie at infinity.
Nevertheless, as we show in Theorem~\ref{thm:Milicevic}, these $10$ points lie on a cubic curve.
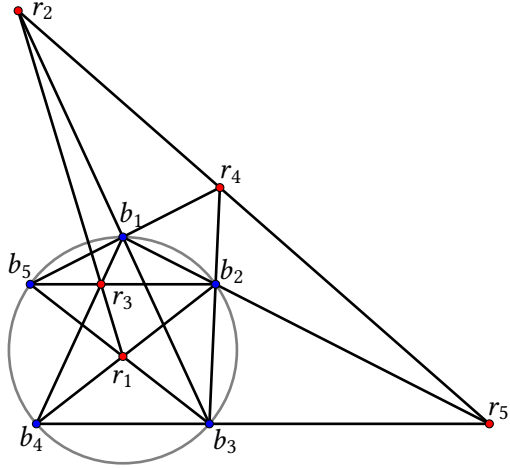
\begin{figure}
\centering
\begin{tikzpicture}[line cap=round,line join=round,scale=1.5]
\draw [line width=1pt,color=gray] (0,0) circle (1cm);
\draw [line width=1pt] (0,1)-- (-0.758795365752062,-0.6513290972413216);
\draw [line width=1pt] (-0.758795365752062,-0.6513290972413216)-- (0.812928447459954,0.5823635799999421);
\draw [line width=1pt] (0.812928447459954,0.5823635799999421)-- (-0.8129284474599541,0.5823635799999418);
\draw [line width=1pt] (-0.8129284474599541,0.5823635799999418)-- (0.7587953657520621,-0.6513290972413212);
\draw [line width=1pt] (-0.9190116821894457,3)-- (0,-0.05572809000084123);
\draw [line width=1pt] (-0.9190116821894457,3)-- (3.2143087503375507,-0.6513290972413204);
\draw [line width=1pt] (-0.8129284474599541,0.5823635799999418)-- (0.8504245592219606,1.4368997905084178);
\draw [line width=1pt] (0.7587953657520621,-0.6513290972413212)-- (0.8504245592219606,1.4368997905084178);
\draw [line width=1pt] (0,1)-- (3.2143087503375507,-0.6513290972413204);
\draw [line width=1pt] (-0.758795365752062,-0.6513290972413216)-- (3.2143087503375507,-0.6513290972413204);
\draw [line width=1pt] (0.7587953657520621,-0.6513290972413212)-- (-0.9190116821894457,3);
\draw [fill=blue] (0.7587953657520621,-0.6513290972413212) circle (1pt);
\draw[color=black] (0.9,-0.8) node {$b_{3}$};
\draw [fill=blue] (-0.758795365752062,-0.6513290972413216) circle (1pt);
\draw[color=black] (-0.8,-0.78) node {$b_{4}$};
\draw [fill=blue] (-0.8129284474599541,0.5823635799999418) circle (1pt);
\draw[color=black] (-0.9,0.75) node {$b_5$};
\draw [fill=blue] (0,1) circle (1pt);
\draw[color=black] (0.08,1.2) node {$b_1$};
\draw [fill=blue] (0.812928447459954,0.5823635799999421) circle (1pt);
\draw[color=black] (0.96,0.6997036359775891) node {$b_2$};
\draw [fill=red] (-0.9190116821894457,3) circle (1pt);
\draw[color=black] (-0.7,3) node {$r_2$};
\draw [fill=red] (-0.19190637444391553,0.5823635799999419) circle (1pt);
\draw[color=black] (0.0,0.44) node {$r_3$};
\draw [fill=red] (0,-0.05572809000084123) circle (1pt);
\draw[color=black] (0.0,-0.25) node {$r_1$};
\draw [fill=red] (0.8504245592219606,1.4368997905084178) circle (1pt);
\draw[color=black] (0.95,1.556317364578696) node {$r_4$};
\draw [fill=red] (3.2143087503375507,-0.6513290972413204) circle (1pt);
\draw[color=black] (3.3,-0.5349686655815484) node {$r_5$};
\end{tikzpicture}
\caption{Five blue points on a circle and five non-collinear red points such that the line through any two blue points contains a red point}\label{fig:10points}
\end{figure}

We show Theorem~\ref{thm:conic-line} by using a special case of Mili\'cevi\'c's conjecture proved by Keller and Pinchasi.
\begin{theorem}[Keller and Pinchasi  \cite{keller2020sets}]\label{thm:KP}
Let $n$ blue points, no three on a line, and $n$ red points, disjoint from the blue points, be given.
If the line through any two blue points contains a red point not between the two blue points, then all $2n$ points lie on a cubic curve.
\end{theorem}

Theorem~\ref{thm:conic-line} will then follow from the following lemma, which shows that the hypotheses of Theorem~\ref{thm:KP} will be satisfied if we assume that the blue points lie on a circle.
%\sujoy{can we add some connectors here?}

\begin{lemma}\label{lemma:3}
Let $n\geq 6$ blue points on a circle, and $n$ red points, disjoint from the blue points, be given.
If the line through any two blue points contains a red point, then the red points all lie outside the circle.
\end{lemma}
\begin{proof}
Denote the set of blue points by $B$ and the set of red points by $R$.
We define a bipartite graph $G$ with parts $B$ and $R$ by connecting $b\in B$ and $r\in R$ if the line through $b$ and $r$ does not contain any other blue points.

We first show that each $b\in B$ has degree at most $1$ in $G$.
Consider the $n-1$ lines $bb'$ where $b'\in B\setminus\{b\}$.
On each of these lines there is a distinct red point.
Thus, there is at most one red point $r$ such that $br$ contains no other blue point.
In other words, $b$ is connected to at most one red point in~$G$.

We now distinguish between two cases depending on the parity of $n$.

The case of odd $n$ is simpler.
We claim that in this case, each $r\in R$ has degree at least $1$ in $G$.
Note that each line through $r$ and a blue point contains either $1$ or $2$ blue points.
Since $n$ is odd, the number of these lines with one blue point is odd, so there is at least one.

It now follows that $G$ is a perfect matching (still when $n$ is odd).
Label the blue points $b_1,b_2,\dots,b_n$ in the order that they appear on the circle.
We consider indices to be modulo $n$.
For each $i=1,\dots,n$, let $r_i$ be the red point matched to $b_i$ in $G$.
Fix an index $i$.
Let $r_j$ be a red point on $b_{i-1}b_{i+1}$, and suppose that $r_j$ is between $b_{i-1}$ and $b_{i+1}$.
Since all lines through $r_j$ and a blue point contain two blue points except for $b_jr_j$, there can only be two more blue points apart from $b_{i-1}, b_i, b_{i+1}$.
This contradicts the assumption $n\geq 6$.
It follows that $r_j$ must be outside the circle, and $b_j=b_i$.
This shows that each red point $r_i$ lies on the line $b_{i-1}b_{i+1}$ outside the circle.
Thus all red points are outside the circle.
This finishes the case where $n$ is odd.

We now assume that $n$ is even.
Again, since each line through an $r\in R$ and a blue point contains either $1$ or $2$ blue points, it now follows that the number of lines with one blue point is even.
Thus, each $r\in R$ has even degree in $G$.
Since we have already shown that $G$ has at most $n$ edges, it follows that the number of red points with non-zero degree in $G$ is at most $n/2$.
Thus the total number of red points with zero degree is at least $n/2$.
As in the odd case, it follows from the assumption $n\geq 6$ that any red point $r$ on a line $b_{i-1}b_{i+1}$ is outside the circle.
Since there is only one blue point on $rb_i$, the degree of $r$ in $G$ is non-zero.

Furthermore, a red point can belong to at most two lines of the form $b_{i-1}b_{i+1}$, so as we go through all $n$ values of $i$, we obtain at least $n/2$ points of non-zero degree in this way.
Therefore, there are exactly $n/2$ points of non-zero degree, with each of degree exactly $2$, and each lying on two lines of the form $b_{i-1}b_{i+1}$ outside the circle.

It also follows then that there are exactly $n/2$ points of degree $0$ in $G$.
Fix an $i$ and consider a red point $r$ on $b_ib_{i+1}$.
If $r$ is inside the circle, then as before we obtain a contradiction.
Thus all of these points are outside the circle.
If one of them has non-zero degree, it has to lie on two lines of the form $b_{i-1}b_{i+1}$ as well as one of the form $b_ib_{i+1}$, which is not possible.
Thus they all have degree zero in $G$.
Since a red point can belong to at most two lines of the form $b_ib_{i+1}$, it follows that there are at least $n/2$ red points on the lines $b_ib_{i+1}$, all of degree $0$.
Thus all degree zero points are also outside the circle.
This concludes the proof for $n$ even.
\end{proof}
Using Keller and Pinchasi's proof, it is easy to see that in the above situation furthermore the $n$ blue points are projectively equivalent to the vertices of a regular $n$-gon on a circle.
\begin{corollary}\label{conj:milicevic-special}
Given a set of $n\geq 6$ blue points on a conic, and a set of $n$ red points, disjoint from the blue points, such that the line through any two blue points contains a red point, then it follows that the red points lie on a line, and after some projective transformation, the line is the line at infinity, the conic is a circle, and the blue points form the vertex set of a regular $n$-gon.
\end{corollary}

\section{Expected number of monochromatic lines}
In this section, we consider random two-colorings of a set of points.
If we have $n$ points with no three on a line, then there are $\binom{n}{2}$ lines.
If each point is independently and uniformly colored red or blue, then the probability that a line is monochromatic is $1/2$, which gives that the expected number of monochromatic lines is $\frac12\binom{n}{2}$.
This is clearly the maximum for a given $n$.
On the other hand, the minimum expected number is trivially $2^{-n+1}$ if all points are collinear.
We are thus interested in the minimum expected number of monochromatic lines in a randomly two-colored set of $n$ points that are not all on a line.
Note that the \emph{near-pencil on $n$ points}, which is any set of $n$ points with exactly $n-1$ points on a line (Figure~\ref{fig:near-pencil}) has an expected number of monochromatic lines of $(n-1)/2 + 2^{-n+2}$.
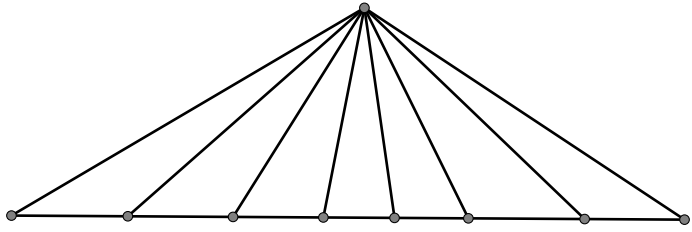
\begin{figure}
\centering
\begin{tikzpicture}[line cap=round,line join=round,scale=0.9]
\draw [line width=1pt] (-6.15,2.11)-- (3.69,2.05);
\draw [line width=1pt] (-6.15,2.11)-- (-0.99,5.17);
\draw [line width=1pt] (-4.448355950477748,2.0996241216492546)-- (-0.99,5.17);
\draw [line width=1pt] (-2.9102424062163066,2.0902453805257095)-- (-0.99,5.17);
\draw [line width=1pt] (-1.5902914823214485,2.0821968992824478)-- (-0.99,5.17);
\draw [line width=1pt] (-0.5497204149161617,2.075851953749489)-- (-0.99,5.17);
\draw [line width=1pt] (0.5299955385358962,2.069268319886976)-- (-0.99,5.17);
\draw [line width=1pt] (2.229810387775588,2.0589035951964902)-- (-0.99,5.17);
\draw [line width=1pt] (3.69,2.05)-- (-0.99,5.17);
\draw [draw=black,fill=gray] (-6.15,2.11) circle (2pt);
\draw [draw=black,fill=gray] (3.69,2.05) circle (2pt);
\draw [draw=black,fill=gray] (-0.99,5.17) circle (2pt);
\draw [draw=black,fill=gray] (-4.448355950477748,2.0996241216492546) circle (2pt);
\draw [draw=black,fill=gray] (-2.9102424062163066,2.0902453805257095) circle (2pt);
\draw [draw=black,fill=gray] (-1.5902914823214485,2.0821968992824478) circle (2pt);
\draw [draw=black,fill=gray] (-0.5497204149161617,2.075851953749489) circle (2pt);
\draw [draw=black,fill=gray] (2.229810387775588,2.0589035951964902) circle (2pt);
\draw [draw=black,fill=gray] (0.5299955385358962,2.069268319886976) circle (2pt);
\end{tikzpicture}
\caption{Near-pencil: $n$ non-collinear points with exactly $n-1$ on a line}\label{fig:near-pencil}
\end{figure}
For $n=7$ there is an example with a smaller expectation:
Consider the vertices, midpoints of edges, and centroid of a triangle.
For this set of $7$ points, the expectation is exactly $3$.
We show that if $n$ is sufficiently large, the near-pencil has the minimum expectation among all non-collinear sets of $n$ points.
We show a more general version where blue and red are not necessarily equiprobable.

\begin{theorem}\label{thm:random}
Let $0\leq p\leq 1$ and $n\geq 10$.
If each of the points of a non-collinear set of $n$ points is colored uniformly and independently at random blue with probability $p$ and red with probability $1-p$, then the expected number of monochromatic lines is uniquely minimized by the near-pencil on $n$ points.
\end{theorem}

To show this theorem, we need two inequalities.
The first is an elementary consequence of the Euler formula, and in dual form was first shown by Melchior \cite{melchior1941vielseite}.
\begin{lemma}[Melchior inequality]
In a set $S$ of $n$ points in the real plane, if we let $t_i$ denote the number of lines containing exactly $i$ points of $S$, then we have
\[\sum_{i\geq 2} (3-i)t_i\geq 3.\]
\end{lemma}
The following inequality, due to Langer \cite[Proposition 11.3.1]{langer2003}, does not have an elementary proof.
See \cite{deZeeuw-survey} for an overview of consequences of this inequality.
\begin{lemma}[Langer inequality]\label{lemma:Langer}
In a set $S$ of $n$ points in the plane (real or complex) with at most $2n/3$ on a line, if we let $t_i$ denote the number of lines containing exactly $i$ points of $S$, then we have
\[\sum_{i\geq 2} it_i\geq n(n+3)/3.\]
\end{lemma}

\begin{proof}[Proof of Theorem~\ref{thm:random}]
If $p=0$ or $1$, then the theorem states that the total number of lines in a non-collinear set of $n$ points is uniquely minimized by the near-pencil, which is a well-known result by De Bruijn and Erd\H{o}s \cite{dBE48}.
Fix $p\in(0,1)$, and let $f(x)=p^x+(1-p)^x$, $x\in\mathbb{R}$.
If we color each point blue with probability $p$ and red with probability $1-p$, then the expected number of monochromatic lines is $\sum_{i\geq 2} f(i)t_i$.
For the near-pencil, this expectation is $(n-1)f(2)+f(n-1)$,
so we would like to prove that for all $n\geq 10$,
\begin{equation}\label{eq:random}
\sum_{i\geq 2} f(i)t_i\geq (n-1)f(2)+f(n-1).
\end{equation}

We first assume that there are at most $2n/3$ points on a line.
We add $(f(2)-\frac23f(3))$ times Melchior's inequality to $\frac13f(3)$ times Langer's inequality to obtain
\[\sum_{i\geq 2} ((3-i)f(2)+(i-2)f(3))t_i\geq 3f(2)-2f(3)+\frac{n(n+3)}{9}f(3).\]
To show \eqref{eq:random}, it is now sufficient to show the following two inequalities for each $i\geq 2$:
\begin{equation}\label{eq:random2}
f(i) \geq (3-i)f(2)+(i-2)f(3)
\end{equation}
and
\begin{equation}\label{eq:random3}
3f(2)-2f(3)+\frac{n(n+3)}{9}f(3) \geq (n-1)f(2)+f(n-1).
\end{equation}
Note that $f(x)$, being the sum of two exponential functions, is convex, hence
\[ f(3) \leq \frac{i-3}{i-2}f(2)+\frac{1}{i-2} f(i),\]
which immediately implies \eqref{eq:random2}.
Since \eqref{eq:random3} clearly holds when $n=3$, we may assume that $n\geq 4$, and then, since $f$ is decreasing, \eqref{eq:random3} will follow from 
\[\left(\frac{n(n+3)}{9}-3\right)f(3) \geq (n-4)f(2),\]
and since $f(2)/f(3)\geq 2$, this will in turn follow from 
\[ \frac{n(n+3)}{9}-3 \geq 2(n-4),\]
which can be checked to hold for all $n\geq 11$.
That \eqref{eq:random3} holds for $n=10$ can be checked separately.
This finishes the case where there are at most $2n/3$ points on a line.

Next suppose that there is a line that contains more than $2n/3$ points.
If we do not have a near-pencil, then the number of ordinary lines is bounded below by $2n-6$ \cite[Lemma 1]{erdospurdy1978}, which gives a lower bound for the expectation of $(2n-6)f(2)$, and we need to show that
\[ (2n-6)f(2)\geq (n-1)f(2) + f(n-1),\]
or equivalently,
\[ (n-5)f(2)\geq f(n-1),\]
which holds for all $n\geq 6$ because $f$ is decreasing.

Thus, the near pencil is the unique minimizer for all $n\geq 10$.
\end{proof}

\section{Monochromatic Circles} 
The Motzkin--Rabin Theorem guarantees a monochromatic line for any non-collinear bicolored set of points in the plane.
Just as there is an analogue of the Sylvester--Gallai Theorem for circles \cite{elliott1967}, we can ask for an analogue of the Motzkin--Rabin Theorem for circles.
\begin{theorem}[Elliott, 1967]
For any finite set of points in the plane not all on a circle or a line, there exists an ordinary circle, that is, a circle through exactly three points of the set.
\end{theorem}
\begin{question}
Given a set of red and blue points in $\mathbb{R}^2$ such that neither color class is on a line or a circle, does there exist a monochromatic circle (including a line as a special circle) passing through at least $3$ points?
\end{question}
The following negative examples show that we would need to make quite strong assumptions on the sets.

\begin{example}
Consider two circles $R$ and $B$ that intersect in two points $r$ and $b$.
For the blue points, choose the point $b$ and some more points on $B$, and for the red points, choose the point $r$ and some more points on $R$.
Then any three blue points determine the circle $B$, and any three red points determine the circle $R$.
Neither of these circles is monochromatic.
Note that each color class lies on a circle.
\end{example}
In the next example, one of the color classes does not lie on a circle.
\begin{example}
Fix a circle $R$ and choose some red points on $R$.
Then independently perturb infinitesimally each of these red points %by a small enough amount 
such that any three red points determine a circle that is very close to $R$.
Let one of the red points $r_1$ still be on the original circle $R$.
Next, choose a circle $B$ that intersects the circle $R$ in two distinct points, one of them being the red point $r_1$,
and the other intersection point $p$
(Figure~\ref{fig:2}).
\begin{figure}
\centering
\newcommand{\pointradius}{1.8pt}
\begin{tikzpicture}[line cap=round,line join=round,scale=0.8]
\draw [line width=1pt,color=gray!70!white] (0.4999932158767801,0.8771701556710643) circle (3.909356213824249cm);
\draw [line width=1pt,color=gray!70!white] (0.14891815856420174,0.8565344705644589) circle (4.108778898724375cm);
\draw [line width=1pt,color=gray!70!white] (-0.2182873264290806,0.834950662202383) circle (4.3381098582056445cm);
\draw [line width=1pt,color=gray!70!white] (0.9436340888330413,-0.3323088722767473) circle (4.849816976828116cm);
\draw [line width=1pt,color=gray!70!white] (0.6540212982051785,0.4572500206081628) circle (4.215085541359776cm);
\draw [line width=1pt,color=gray!70!white] (-0.30067642941570294,0.08682730229125611) circle (3.587074879385342cm);
\draw [line width=1pt,color=gray!70!white] (-0.14464860772423968,0.034270562353078356) circle (3.743517960110128cm);
\draw [line width=1pt,color=gray!70!white] (-0.021658142183904244,0.19508639773719538) circle (3.7526279215090272cm);
\draw [line width=1pt,color=gray!70!white] (-0.051991774331939154,0.2937971120061331) circle (3.8411612651608134cm);
\draw [line width=1pt,color=gray!70!white] (-0.30067642941570294,0.08682730229125611) circle (3.587074879385342cm);
\draw [line width=1pt,color=blue!70!white] (4.557537298097213,-0.7145663617494188) circle (2.3716663679885572cm);
\draw [line width=1pt,color=gray!70!white] (0.3188046081023837,0.6023953915805089) circle (4.237630917995001cm);
\draw [line width=1pt,color=red] (0.11,0.6) circle (3.9419280013117444cm);
\draw[color=red] (1.1,4.1) node {$R$};
\draw [fill=red] (2.7874798111753436,-2.293077639517137) circle (\pointradius) node[right] {$r_1$};
\draw [draw=black,fill=white] (3.93,1.57) circle (1.5pt) node[above left] {$p$};
\draw [fill=red] (2.400326056893482,4.29357193238288) circle (\pointradius);
\draw [fill=red] (-3.1586975165129654,2.254503859350623) circle (\pointradius);
\draw [fill=red] (-3.8877497884351224,0.09012992708174661) circle (\pointradius);
\draw [fill=red] (-0.9487578172489268,-3.4412170150411567) circle (\pointradius);
\draw[color=blue] (2.3,0.6821170531893088) node {$B$};
\draw [draw=black,fill=blue!70!white] (5.43923312971869,1.4871179529399854)   circle (\pointradius);
\draw [draw=black,fill=blue!70!white] (4.695925162854295,1.6530590646928667)  circle (\pointradius);
\draw [draw=black,fill=blue!70!white] (4.424048874733936,1.653340351336217)   circle (\pointradius);
\draw [draw=black,fill=blue!70!white] (4.332915267805289,1.646438997108448)   circle (\pointradius);
\draw [draw=black,fill=blue!70!white] (4.184696430158688,1.6276102003916344)  circle (\pointradius);
\draw [draw=black,fill=blue!70!white] (3.118700871820044,1.1707853559589732)  circle (\pointradius);
\draw [draw=black,fill=blue!70!white] (4.051387790666126,1.6024606206518017)  circle (\pointradius);
\draw [draw=black,fill=blue!70!white] (3.3653713559440597,1.3356877023651053) circle (\pointradius);
\draw [draw=black,fill=blue!70!white] (3.5252056090413215,1.420636824090057)  circle (\pointradius);
\draw [draw=black,fill=blue!70!white] (3.608231608808427,1.4588227127910742)  circle (\pointradius);
\end{tikzpicture}
    \caption{Five red points not lying on a line or a circle and $\binom{5}{2}$ blue points with no monochromatic circle}\label{fig:2}
\end{figure}
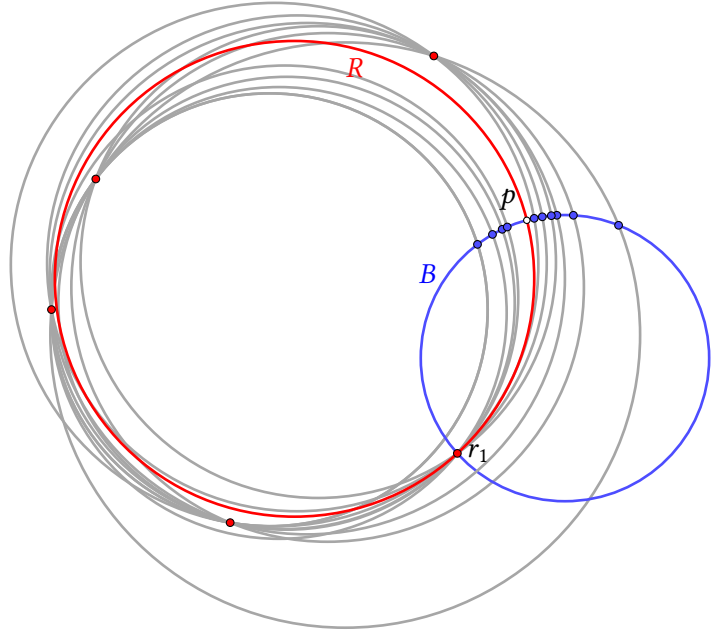
%\sujoy{should we highlight in the fig.? also can we enlarge the scale of the points a bit}. \konrad{I made the points a bit bigger.}
%Call the other intersection point $r_2$.
Then $B$ also intersects each of the circles through triples of red points in two points, one close to $r_1$, and one close to $p$.
These points are all distinct.
Choose as blue points the intersection points that are close to $p$.
Then the blue points lie on a unique circle that also contains a red point, so there is no blue monochromatic circle.
Also, each circle through three red points contains a blue point.
Thus we have found a set in which one color class is not on a circle, and with no monochromatic circle.
\end{example}
\begin{example}
Consider two concentric circles $B$ and $R$ centered at the origin of the plane.
On each circle, we can determine a point uniquely by specifying its angle from the positive $x$-axis.
Then two points on $B$ with angles $\alpha_1$ and $\alpha_2$, and two points on $R$ with angles $\beta_1$ and $\beta_2$, lie on a circle iff $\alpha_1+\alpha_2=\beta_1+\beta_2$.

Let $m$ be even.
On circle $B$, choose $m$ blue points with angles $\frac{2\pi k}{m}$, $k=0,1,\dots,m-1$ and $m$ red points with angles $\frac{2\pi k}{m}-\pi$, $k=0,1,\dots,m-1$.
On circle $R$, choose $m$ red points with angles $\frac{2\pi k}{m} -\frac{\pi}{2}$ and $m$ blue points with angles $\frac{2\pi k}{m}+\frac{\pi}{2}$, $k=0,1,\dots,m-1$.
Using the above criterion on when two points from $B$ and two points from $R$ are on a circle, it can then be checked that there is no monochromatic circle.
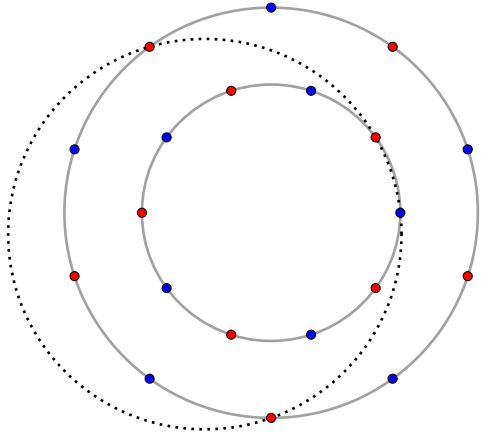
\begin{figure}
    \centering
\definecolor{aqaqaq}{rgb}{0.63,0.63,0.63}
\begin{tikzpicture}[scale=1.7]
\draw [line width=1pt,color=aqaqaq] (0,0) circle (1cm);
\draw [line width=1pt,color=aqaqaq] (0,0) circle (1.6cm);
\draw [line width=1pt,dotted] (-0.5132012757849717,-0.1667492026343929) circle (1.5223611258884904cm);
\draw [fill=blue] (1,0) circle (1pt);
\draw [fill=blue] (0.30901699437494745,0.9510565162951535) circle (1pt);
\draw [fill=blue] (-0.8090169943749473,0.5877852522924732) circle (1pt);
\draw [fill=blue] (-0.8090169943749476,-0.587785252292473) circle (1pt);
\draw [fill=blue] (0.30901699437494723,-0.9510565162951536) circle (1pt);
\draw [fill=blue] (0,1.6) circle (1pt);
\draw [fill=blue] (-1.5216904260722457,0.49442719099991606) circle (1pt);
\draw [fill=blue] (-0.9404564036679572,-1.2944271909999159) circle (1pt);
\draw [fill=blue] (0.9404564036679567,-1.294427190999916) circle (1pt);
\draw [fill=blue] (1.521690426072246,0.4944271909999155) circle (1pt);
\draw [fill=red] (-1,0) circle (1pt);
\draw [fill=red] (-0.30901699437494734,-0.9510565162951536) circle (1pt);
\draw [fill=red] (0.8090169943749475,-0.5877852522924731) circle (1pt);
\draw [fill=red] (0.8090169943749475,0.5877852522924731) circle (1pt);
\draw [fill=red] (-0.30901699437494734,0.9510565162951536) circle (1pt);
\draw [fill=red] (0,-1.6) circle (1pt);
\draw [fill=red] (1.5216904260722457,-0.49442719099991583) circle (1pt);
\draw [fill=red] (0.940456403667957,1.294427190999916) circle (1pt);
\draw [fill=red] (-0.9404564036679569,1.294427190999916) circle (1pt);
\draw [fill=red] (-1.521690426072246,-0.49442719099991567) circle (1pt);
\end{tikzpicture}
\caption{Ten red points not on a circle and ten blue points not on a circle with no monochromatic circle}
\end{figure}
\end{example}
In the above example, the union of two circles is an algebraic curve of degree $4$, though not irreducible.
As we now show, there also exist irreducible algebraic curves of degree $4$ (and $2$ and $3$) of red and blue points with no monochromatic circle.
\begin{example}\label{ex:ellipse}
Let $m$ be odd.
On the ellipse with equation $x^2+4y^2=1$, take the $m$ blue points $\left(\cos(\frac{2\pi k}{m}-\frac{\pi}{4}),2\sin(\frac{2\pi k}{m}-\frac{\pi}{4})\right)$, $k=0,\dots,m-1$, and the $m$ red points $\left(\cos(\frac{2\pi k}{m}+\frac{3\pi}{4}),2\sin(\frac{2\pi k}{m}+\frac{3\pi}{4})\right)$, $k=0,\dots,m-1$.
(The blue and red points are disjoint because $m$ is odd.)
Then it can easily be checked that there is no monochromatic circle (Figure~\ref{fig:ellipse}).
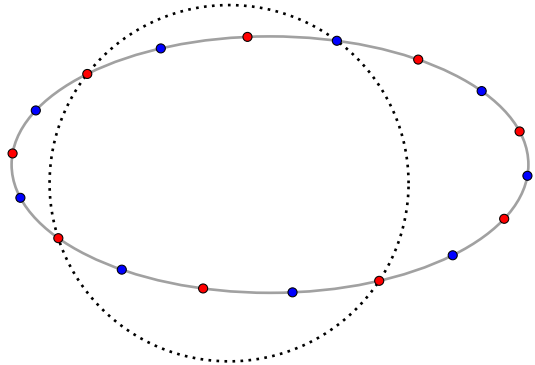
\begin{figure}[h]
    \centering
\definecolor{aqaqaq}{rgb}{0.63,0.63,0.63}
\begin{tikzpicture}[scale=3.4]
\draw [rotate around={0:(0,0)},line width=1pt,color=aqaqaq] (0,0) ellipse (1cm and 0.5cm);
\draw [line width=1pt,dotted] (-0.15840403474356013,-0.07243064403296995) circle (0.6946489053547339cm);
\draw [fill=blue] (0.7071067811865476,-0.35355339059327373)  circle (0.5pt);
\draw [fill=blue] (0.9961946980917455,-0.04357787137382909)  circle (0.5pt);
\draw [fill=blue] (0.8191520442889918,0.286788218175523)     circle (0.5pt);
\draw [fill=blue] (0.25881904510252096,0.4829629131445341)   circle (0.5pt);
\draw [fill=blue] (-0.42261826174069933,0.453153893518325)   circle (0.5pt);
\draw [fill=blue] (-0.9063077870366499,0.21130913087034975)  circle (0.5pt);
\draw [fill=blue] (-0.9659258262890684,-0.12940952255126018) circle (0.5pt);
\draw [fill=blue] (-0.5735764363510464,-0.4095760221444958)  circle (0.5pt);
\draw [fill=blue] (0.08715574274765789,-0.4980973490458728)  circle (0.5pt);
\draw [fill=red] (-0.7071067811865475,0.3535533905932738)    circle (0.5pt);
\draw [fill=red] (-0.9961946980917455,0.0435778713738291)    circle (0.5pt);
\draw [fill=red] (-0.8191520442889918,-0.2867882181755231)   circle (0.5pt);
\draw [fill=red] (-0.25881904510252063,-0.48296291314453416) circle (0.5pt);
\draw [fill=red] (0.42261826174069883,-0.45315389351832513)  circle (0.5pt);
\draw [fill=red] (0.90630778703665,-0.2113091308703496)      circle (0.5pt);
\draw [fill=red] (0.9659258262890684,0.12940952255126012)    circle (0.5pt);
\draw [fill=red] (0.5735764363510465,0.40957602214449573)    circle (0.5pt);
\draw [fill=red] (-0.08715574274765776,0.4980973490458728)   circle (0.5pt);
\end{tikzpicture} 
\caption{Nine red points and nine blue points on an ellipse with no monochromatic circle}\label{fig:ellipse}
\end{figure}
\end{example}
\begin{example}
Consider a circular cubic curve $\gamma$ in the real projective plane.
This is defined to be a cubic curve defined by a homogeneous polynomial of degree $3$ of the form $(ux + vy)(x^2 + y^2) + q(x, y, z)z$, where $u, v\in\mathbb{R}$, and $q(x,y,z)$ a quadratic homogeneous polynomial.
Equivalently, a circular cubic is a cubic curve that intersects the line at infinity at the circular points at infinity that lie on all circles.
As described in \cite{LMMSSZ}, there is an abelian group $(\gamma^*,\oplus)$ on the set $\gamma^*$ of non-singular points of this curve with zero element $o$ chosen to be the real point of intersection of $\gamma^*$ with the line at infinity, and such that $4$ points lie on a circle iff their sum in the group is equal to the real point $\omega$ where the tangent to $\gamma$ at $o$ intersects $\gamma$ again.
If $\gamma$ is either smooth or acnodal, then the group on $\gamma^*$ has
finite cyclic subgroups of any order $n$  (see \cite{LMMSSZ} for a description of these curves).

Let $m$ be odd, let $g,h\in\gamma^*$ be elements of order $m$ and $8$, respectively, in $\gamma^*$.
Also, let $\omega/4$ be an element of $\gamma^*$ such that $4\cdot(\omega/4)=\omega$.
Let $B=\{k\cdot g - h + \omega/4 : k\in\mathbb{Z}\}$ be the set of blue points and $R=\{k\cdot g + 3\cdot h +\omega/4 : k\in\mathbb{Z}\}$ the set of red points.
Note that $|B|=|R|=m$, and that $B$ and $R$ are disjoint because $m$ is odd.
Then for any three blue points there exists a red point such that the sum of the three blue points and the red point is equal to $\omega$, giving that the four points lie on a circle.
The same statement holds if the colors are interchanged, so in this set of blue and red points, there is no monochromatic circle.

Examples that lie on bicircular quartics can now be found by applying an inversion of the circular cubic in a point not lying on the curve.
Also, if an acnodal cubic is used in the above example, and the cubic is inverted in its singularity, then the curve transforms into an ellipse, and we regain Example~\ref{ex:ellipse}.
\end{example}
After considering these examples, we may ask what conditions should be imposed on the set of points so that there will be a monochromatic circle if we two-color the points.
More specifically, we can ask the following question.

\begin{question}
    Suppose that in a set of $n$ points in the plane not on a line or a circle there is no monochromatic circle.
    Does it follow that either one of the colors lie on a circle, or all the points lie, up to inversion, on one of the curves in the above examples, namely
    an ellipse, the union of two disjoint circles, a circular elliptic cubic, a circular acnodal cubic, a bicircular elliptic quartic, or a bicircular elliptic acnodal quartic?
\end{question}
We could also replace circle with some other strictly convex closed curve, and then ask if we obtain a monochromatic homothet (translate of a scaled copy) of the curve.
\begin{question}
For which closed strictly convex curves $C$ is the following true:
Given a finite set $B\cup R$ of red and blue points, not all lying on a homothet of $C$, is there a homothet of $C$ that intersects $B\cup R$ in at least $3$ points, all of the same color? 

Is this true for ``almost all'' convex curves?
\end{question}

\section{Monochromatic Conics}
Just as for circles, there are examples of sets of blue and red points on cubic curves, with neither color class on a conic, and such that there is no monochromatic conic.
We consider a monochromatic conic to be a conic passing through at least $5$ points of the set, and such that all points of the set on the conic are of the same color.
\begin{example}
Consider a planar cubic curve $\gamma$ , and let $(\gamma^*,\oplus)$ be the group on its non-singular points such that three points $a,b,c\in\gamma^*$ are collinear iff $a\oplus b\oplus c=0$ in the group.
It is known that $6$ points of $\gamma^*$ lie on a (not necessarily irreducible) conic iff their sum in the group is $0$.
If $\gamma$ is chosen to be smooth or acnodal, the group $\gamma^*$ contains a cyclic subgroup $G_n$ of order $n$ for any value of $n$.
Choose $n=24k$ for some $k\geq 1$, and let $H$ be a subgroup of $G_n$ of index $24$.
We identify $G$ with $(\mathbb{Z}_n,+)$, and then $H$ is the subgroup generated by $24\in\mathbb{Z}_n$.
We choose the following cosets for the blue and red points: $B=5+H$, $R=-1+H$.
Consider a conic that passes through $5$ blue points $b_1,b_2,b_3,b_4,b_5$.
By the theorem of B\'ezout, the conic intersects $\gamma^*$ in $6$ points, which have to be $b_1,\dots,b_5$ and $-(b_1+\dots+b_5)\in -25+H=-1+H=R$.
Thus the conic also passes through a red point.
Similarly, if a conic passes through $5$ red points, then the $6$th point has to be in $5+H=B$, so has to blue.
Thus, there are no monochromatic conics in $B\cup R$, which contains $k$ blue points and $k$ red points, with neither $B$ nor $R$ contained in a conic, since a conic intersects $\gamma$ in at most $6$ distinct points.
\end{example}
Are there any positive results by making stronger assumptions on the given point set?
\begin{question}
Consider a finite set $S$ of points in the plane, each colored blue or red.
Suppose that $S$ does not lie on a cubic curve.
Does there exist a conic that intersects $S$ in at least $5$ points, all of the same color?
\end{question}

\section{Other Related Work}\label{related-work}
We conclude by giving an overview of the other related work, which may lead to interesting directions in connection with the attempt of giving quantitative bounds on the Motzkin-Rabin theorem. 

\paragraph{Bichromatic Ordinary Lines:} A natural variant concerns \emph{bichromatic ordinary lines}, namely lines containing exactly one red and one blue point. Pach and Pinchasi~\cite{PachP00} showed that such lines need not exist. On the other hand, they proved that there always exists a line containing at most two red and at most two blue points. This extends earlier conjectures of Fukuda~\cite{da1998isolating,PachP00}, who asked whether a bichromatic ordinary line must exist when the two colors are separated by a line and their cardinalities differ by at most one; Pach and Pinchasi confirmed the statement without the cardinality restriction.

\vspace{-.3cm}
\paragraph{Sylvester-Gallai and related problems for Other Shapes:} 
A natural direction is to seek Sylvester--Gallai--type statements in which lines are replaced by algebraic curves of higher degree.
Elliott~\cite{elliott1967} initiated the study of such questions for other families of curves by proving an analogue of the Sylvester--Gallai theorem for circles: if a finite set of points in $\mathbb{R}^2$ is not contained in a single line or a single circle, then there exists a circle passing through exactly three of the points. The number three here is significant---it is precisely the number of points that typically determine a circle. In general, one expects a Sylvester--Gallai--type statement for a given family of curves to yield a curve passing through exactly as many points as are normally required to determine one, such a curve being called \emph{ordinary}. Thus, an ordinary line is determined by two points, an ordinary circle by three, and similarly, an ordinary conic (as shown by Wiseman and Wilson~\cite{wiseman1988sylvester}) by five points, since five generically determine a conic (not necessarily irreducible). Recently, Cohen and Zeeuw~\cite{cohen2022sylvester} proved a variant of the Sylvester–Gallai theorem for cubics (algebraic curves of degree three): If a finite set of sufficiently
many points in $\mathbb{R}^2$
is not contained in a cubic, then there is
a cubic that contains exactly nine of the points.

\vspace{-.3cm}
\paragraph{Fractional Sylvester-Gallai and Motzkin-Rabin:}
The Sylvester--Gallai (SG) theorem and its colorful counterpart, the Motzkin--Rabin theorem, have deep modern connections to the theory of locally correctable and locally decodable codes (LCCs and LDCs).  In both settings, the essential phenomenon is that of \emph{local linear dependencies}: in the geometric case, every pair of points determines a third collinear one, while in the coding-theoretic formulation, each symbol of a codeword can be recovered from a few others.  This correspondence, first developed by Barak, Dvir, Yehudayoff, and Wigderson~\cite{BarakDYW11} (see also the survey by Dvir~\cite{Dvir12}), gave rise to quantitative SG-type results over general fields through rank bounds on \emph{design matrices}, which implies strong dimensionality limitations for low-query linear LCCs over the reals and complexes.  Dvir and Tessier-Lavigne~\cite{dvir2015quantitative} later proved a quantitative variant of multi-colored Motzkin-Rabin theorem in the spirit of the work of Barak et al.~\cite{BarakDYW11}. 
Subsequent works refined these methods via rigidity and analytic frameworks, tensor and algebraic techniques, and subspace-evasive constructions (see, e.g.,~\cite{DvirH16, bhattacharyya2011tight, dvir2012subspace, ai2014sylvester}). 
Moreover, these geometric--algebraic ideas have influenced advances in \emph{polynomial identity testing} and \emph{algebraic circuit complexity}~\cite{shpilka2019sylvester, saxena2013sylvester, dvir2014improved, oliveira2024strong}. These works established the connection of classical incidence geometry with modern complexity theory.

\section{Acknowledgement}
The first author acknowledges Ajit A. Diwan and G\'{a}bor Tardos for helpful pointers and comments during the early stage of this work.

%\bibliographystyle{alphaabbrv}
%\bibliography{main}

\input{main.bbl}

\end{document}

%% file: main.bbl
\newcommand{\etalchar}[1]{$^{#1}$}